\documentclass[10pt]{amsart}

\usepackage{mathrsfs,amssymb,amsmath}
\usepackage{verbatim}
\usepackage{hyperref}
\usepackage{graphicx}

\usepackage{subfigure}

\usepackage{vmargin}

\usepackage{layout}
\begin{document}


\newtheorem{theorem}{Theorem}
\newtheorem{proposition}{Proposition}
\newtheorem{lemma}{Lemma}
\newtheorem{corollary}{Corollary}
\newtheorem{conjecture}[theorem]{Conjecture}
\newtheorem{question}[theorem]{Question}
\newtheorem{problem}[theorem]{Problem}

\theoremstyle{definition}
\newtheorem{definition}[theorem]{Definition}
\newtheorem{example}[theorem]{Example}

\theoremstyle{remark}
\newtheorem{remark}[theorem]{Remark}

\def\theenumi{\roman{enumi}}

\numberwithin{equation}{section}

\renewcommand{\Re}{\operatorname{Re}}
\renewcommand{\Im}{\operatorname{Im}}

\newcommand{\ind}{1\hspace{-.27em}\mbox{\rm l}}
\newcommand{\inde}{1\hspace{-.23em}\mathrm{l}}
\newcommand{\lqn}[1]{\noalign{\noindent $\displaystyle{#1}$}}
\newcommand{\bm}[1]{\mbox{\boldmath{$#1$}}}
\newcommand{\eps}{\varepsilon}
\newcommand{\lip}{{\rm Lip}}
\newcommand{\R}{\mathbb{R}}
\newcommand{\N}{\mathbb{N}}
\newcommand{\Z}{\mathbb{Z}}
\newcommand{\Sym}{\mathfrak{S}}
\newcommand{\Leb}{{\lambda}}
\newcommand{\md}{\mathbb{D}}
\newcommand{\me}{\mathbb{E}}
\newcommand{\Q}{\mathbb{Q}}
\newcommand{\1}{{\sf 1}}
\newcommand{\tmu}{\tilde{\mu}}
\newcommand{\tnu}{\tilde{\nu}}
\newcommand{\hmu}{\hat{\mu}}
\newcommand{\mb}{\bar{\mu}}
\newcommand{\bnu}{\overline{\nu}}
\newcommand{\tx}{\tilde{x}}
\newcommand{\tz}{\tilde{z}}
\newcommand{\A}{{\mathcal A}}
\newcommand{\skria}{{\mathcal A}}
\newcommand{\skrib}{{\mathcal B}}
\newcommand{\skric}{{\mathcal C}}
\newcommand{\skrid}{{\mathcal D}}
\newcommand{\skrie}{{\mathcal E}}
\newcommand{\skrif}{{\mathcal F}}
\newcommand{\skrig}{{\mathcal G}}
\newcommand{\skrih}{{\mathcal H}}
\newcommand{\skrii}{{\mathcal I}}
\newcommand{\skrik}{{\mathcal K}}
\newcommand{\skril}{{\mathcal L}}
\newcommand{\skrim}{{\mathcal M}}
\newcommand{\skrin}{{\mathcal N}}
\newcommand{\skrip}{{\mathcal P}}
\newcommand{\skriq}{{\mathcal Q}}
\newcommand{\skris}{{\mathcal S}}
\newcommand{\skrit}{{\mathcal T}}
\newcommand{\skriu}{{\mathcal U}}
\newcommand{\skriw}{{\mathcal W}}
\newcommand{\skrix}{{\mathcal X}}
\newcommand{\heap}[2]{\genfrac{}{}{0pt}{}{#1}{#2}}
\newcommand{\sfrac}[2]{ \, \mbox{$\frac{#1}{#2}$}}
 \renewcommand{\labelenumi}{(\roman{enumi})}
\newcommand{\ssup}[1] {{\scriptscriptstyle{({#1}})}}

\newcommand{\cal}{\mathcal}
\allowdisplaybreaks


\def \R {{\mathbb R}}
\def \HH {{\mathbb H}}
\def \N {{\mathbb N}}
\def \C {{\mathbb C}}
\def \Z {{\mathbb Z}}
\def \Q {{\mathbb Q}}
\def \TT {{\mathbb T}}
\newcommand{\T}{\mathbb T}
\def \Dc {{\mathcal D}}

\newcommand{\tr}[1] {\hbox{tr}\left( #1\right)}

\newcommand{\area}{\operatorname{area}}

\newcommand{\Norm}{\mathcal N}
\newcommand{\simgeq}{\gtrsim}%
\newcommand{\simleq}{\lesssim}

\newcommand{\length}{\operatorname{length}}

\newcommand{\curve}{\mathcal C} 
\newcommand{\vE}{\mathcal E} 
\newcommand{\Ec}{\mathcal {E}} 
\newcommand{\Sc}{\mathcal{S}} 

\newcommand{\dist}{\operatorname{dist}}
\newcommand{\supp}{\operatorname{supp}}
\newcommand{\spec}{\operatorname{spec}}
\newcommand{\diam}{\operatorname{diam}}

\newcommand{\Ccap}{\operatorname{Cap}}
\newcommand{\E}{\mathbb E}

\newcommand{\sumstar}{\sideset{}{^\ast}\sum}

\newcommand {\Zc} {\mathcal{Z}} 
\newcommand{\ninumber}{\Zc}

\newcommand{\zeigen}{E} 
\newcommand{\eigen}{m}

\newcommand{\ave}[1]{\left\langle#1\right\rangle} 

\newcommand{\Var}{\operatorname{Var}}
\newcommand{\Prob}{\operatorname{Prob}}

\newcommand{\var}{\operatorname{Var}}
\newcommand{\Cov}{{\rm{Cov}}}
\newcommand{\meas}{\operatorname{meas}}

\newcommand{\leg}[2]{\left( \frac{#1}{#2} \right)}  

\renewcommand{\^}{\widehat}

\newcommand {\Rc} {\mathcal{R}}

\title[Fluctuations of the  Euler-Poincar\'{e} characteristic]
{Fluctuations of the  Euler-Poincar\'{e} characteristic  \\for  
random spherical harmonics}
\author{V. Cammarota, D. Marinucci and I. Wigman}

\address{Department of Mathematics, Universit\`a degli Studi di Roma Tor Vergata } \email{cammarot@mat.uniroma2.it}

\address{Department of Mathematics, Universit\`a degli Studi di Roma Tor Vergata} \email{marinucc@mat.uniroma2.it}

\address{Department of Mathematics, King's College London }
\email{igor.wigman@kcl.ac.uk}

\date{\today}

\thanks{Research Supported by ERC Grant n$^{\text{o}}$ 277742 \emph{Pascal} (V.C., D.M.) and  n$^{\text{o}}$ 335141 (I.W.).
}

\begin{abstract}
In this short note, we build upon recent results from \cite{CMW} to present a precise expression for the asymptotic variance of the Euler-Poincar\'e characteristic for the excursion sets of Gaussian eigenfunctions on ${\cal S}^2$.
\end{abstract}

\maketitle

\section{Introduction and main result}

The geometry of excursion sets for Gaussian random fields has been a subject
of intense research over the last fifteen years; much work has focussed on
the investigation of the Euler-Poincar\'e characteristic, henceforth EPC 
\cite{adlerstflour}. We recall here that the EPC $\chi(A)$ is the unique
integer-valued functional, defined on the ring $\mathcal{C}$ of closed
convex sets in $\mathbb{R}^N$, which equals $\chi(A)=0$ if $A=\emptyset$, $%
\chi(A)=1$ if $A$ is homotopic to the unit ball, and satisfies the
additivity property 
\begin{equation*}
\chi(A \cup B)=\chi(A)+\chi(B)-\chi(A \cap B), \hspace{1cm}\text{for all\;}
A,B \in \mathcal{C}.
\end{equation*}
Clearly, the EPC is a topological invariant (i.e. it is invariant under
homeomorphisms); its investigation for the excursion sets of random fields
was initiated in the late seventies by Robert Adler and his co-authors. This
stream of research has eventually resulted with the discovery of the
beautiful Gaussian Kinematic Formula (GKF) \cite{taylor, adlertaylor}.

More precisely, let $f$ be a real valued random field defined on the
parameter space $\mathcal{M}$; its excursion sets are defined as 
\begin{equation*}
A_{u}(f; {\mathcal{M}})=\left\{ x \in {\mathcal{M}}:f(x)\ge u \right\}, 
\hspace{1cm} u \in{\mathbb{R}}.
\end{equation*}
Let $\mathcal{L}^f_j$'s for $j=0,\dots, \text{dim} (\mathcal{M})$, denotes
the Lipschitz-Killing curvatures for the manifold $\mathcal{M}$ with
Riemannian metric $g^f$ induced by the covariance of $f$, i.e., for $U_x,
V_x \in T_x \mathcal{M}$, the tangent space to $\mathcal{M}$ at $x$ we have 
\begin{equation*}
g^f_x(U_x, V_x):=\mathbb{E}[(U_x f)\cdot (V_x f)],
\end{equation*}
(see \cite{adlertaylor} for further details); in particular $\mathcal{L}_0$
is the EPC. The functions $\rho_j$'s are the so-called Gaussian Minkowski
functionals and they are defined by 
\begin{align}  \label{GKF2}
{\rho}_j(u)= (2 \pi)^{-(j+1)/2} H_{j-1}(u) e^{-u^2/2},
\end{align}
where $H_q(\cdot)$ are the Hermite polynomial of order $q$: 
\begin{equation*}
H_{-1}(u)=1- \Phi(u), \hspace{1cm} H_{j}(u)=(-1)^j (\phi(u))^{-1} \frac{d^j}{%
d u^j} \phi(u), \hspace{0.4cm} j=0,1,\dots,
\end{equation*}
$\phi(\cdot)$, $\Phi(\cdot)$ denote the standard Gaussian density and
distribution functions, respectively; for example: 
\begin{equation*}
H_0(u)=1, \hspace{0.3cm} H_1(u)=u, \hspace{0.3cm} H_2(u)=u^2-1, \hspace{0.3cm%
} H_3(u)=u^3-3u,\dots
\end{equation*}

The GKF states that the expected EPC of the excursion sets of a smooth,
centred, unit variance, Gaussian random fields $f:\mathcal{M} \to \mathbb{R}$
is 
\begin{align}  \label{GKF1}
\mathbb{E}[ \chi(A_{u}(f; {\mathcal{M}}))]=\sum_{j=0}^{\text{dim} (\mathcal{M%
})} {\mathcal{L}}^{f}_j({\mathcal{M}}) {\rho}_j(u).
\end{align}

While the GKF yields a precise expression for the expected value of the EPC
of excursion sets of smooth Gaussian processes, the analysis of higher
moments, and, in particular, of the variance, is still open. The latter
question is of both theoretical and applied interest; for instance, in the
recent paper \cite{adlerkou} five different methods are suggested to
estimate numerically the covariance matrix of the EPC characteristic for the
joint excursion sets at various thresholds. These results were subsequently
exploited to approximate excursion probabilities, the so-called
Euler-Poincar\'e heuristic \cite{adlerstflour} Section 5.1.

In this paper, we establish analytic formulae for the covariance of the EPC
characteristic of excursion sets at different thresholds, focussing on an
important class of fields: Gaussian spherical harmonics. We establish a
rather simple expression which seems to be closely related to a second-order
Gaussian Kinematic formula, in a sense to be made clear below. More
precisely, consider the Laplace equation 
\begin{equation*}
\Delta _{\mathcal{S}^2}f_{\ell}-\lambda_{\ell} f_{\ell}=0, \hspace{1cm}
f_{\ell}: {\mathcal{S}}^2 \to \mathbb{\mathbb{R}},
\end{equation*}
where $\Delta _{\mathcal{S}^2}$ is the Laplace-Beltrami operator on $%
\mathcal{S}^2$ and $\lambda_{\ell}=-\ell(\ell+1)$, $\ell=0,1,2,\dots$. For a
given eigenvalue $\lambda_{\ell}$, the corresponding eigenspace is the $(2
\ell+1)$-dimensional space 
of spherical harmonics of degree $\ell$; we can choose an arbitrary $L^{2}$%
-orthonormal basis $\left\{ Y_{\mathbb{\ell }m}(.)\right\} _{m=-\ell, \dots,
\ell }$, and consider random eigenfunctions of the form 
\begin{equation*}
f_{\ell }(x)=\frac{1}{\sqrt{2\ell +1}}\sum_{m=-\ell }^{\ell }a_{\ell
m}Y_{\ell m}(x),
\end{equation*}
where the coefficients $\left\{ a_{\mathbb{\ell }m}\right\}$ are
independent, standard Gaussian variables. The law of $f_{\ell}$ is invariant
w.r.t. the choice of a $L^2$-orthonormal basis $\{Y_{\ell m}\}$. The random
fields $\{f_{\ell }(x), \; x\in \mathcal{S}^2\}$ are centred, Gaussian and
isotropic, meaning that the probability laws of $f_{\ell }(\cdot )$ and $%
f_{\ell }(g\cdot )$ are the same for any rotation $g\in SO(3)$. From the
addition theorem for spherical harmonics (\cite{AAR} Theorem 9.6.3) the
covariance function is given by 
\begin{equation*}
\mathbb{E}[f_{\ell }(x)f_{\ell }(y)]=P_{\ell }(\cos d(x,y)),
\end{equation*}%
where $P_{\ell }$ are the Legendre polynomials and $d(x,y)$ is the spherical
geodesic distance between $x$ and $y$. An application of the GKF \eqref{GKF1}
gives in these circumstances: 
\begin{align}  \label{expect}
\mathbb{E} [\chi (A_{u}(f_{\ell }; {\mathcal{S}}^2))] =\frac{\sqrt{2}}{ 
\sqrt{\pi}} \exp\{-u^2/2\} u \frac{\ell (\ell+1)}{2}+2 [1-\Phi(u)],
\end{align}
for a proof of formula \eqref{expect} see, for example, \cite{chengxiao_a}
Lemma 3.5 or \cite{MV} Corollary 5. Note that, as $u \to -\infty$, the right
hand side of \eqref{expect} yields the Euler-Poincar\'{e} characteristic of
the two-dimensional sphere i.e. 
\begin{equation*}
\lim_{u \to -\infty} \mathbb{E}[\chi (A_{u}(f_{\ell }; {\mathcal{S}}^2)) ]=2.
\end{equation*}
The analysis of spherical Gaussian eigenfunctions is motivated by
applications arising mainly from Mathematical Physics and Cosmology. In
particular, Gaussian eigenfunction have been conjectured \cite{Berry 1977}
to approximate deterministic eigenfunctions on generic billiards (surfaces
with smooth boundaries). On the other hand, spherical Gaussian
eigenfunctions are the Fourier components of isotropic spherical random
fields, and, because of this, have been deeply exploited in the analysis of
cosmological data, see for instance \cite{MaPeCUP}.

\noindent Let $I \subseteq \mathbb{R}$ be any interval in the real line and 
\begin{equation*}
A_{I}(f_{\ell }; \mathcal{S}^2)=f_{\ell}^{-1}(I)=\left\{ x \in \mathcal{S}%
^2: f_{\ell}(x)\in I \right\}.
\end{equation*}
Our principal result is the following:

\begin{theorem} \label{cov}  As $\ell \to \infty$, for every intervals $I_1, I_2 \subseteq \mathbb{R}$,
\begin{align} \label{11:14_1}
{\rm Cov}[\chi (A_{I_1}(f_{\ell }; {\cal S}^2) ) \,,\, \chi (A_{I_2}(f_{\ell }; {\cal S}^2) ) ]=\frac{\ell^3}{8 \pi} 
{\cal I}_1  {\cal I}_2 +O(\ell^{5/2}),
\end{align}
where  ${\cal I}_{i}$, for $i=1,2$, are given by
\begin{align*}
{\cal I}_{i}=\int_{I_i} p(t_i) d t_i, \hspace{1cm}  p(t)=  (-t^4+4 t^2 -1) e^{-\frac {t^2} 2 }.
\end{align*}
The constant involved in the $O(\cdot)$ notation is universal.
\end{theorem}
In particular the variance for any given interval follows as an easy
corollary: 
\begin{corollary} \label{var}
For every interval $I \subseteq \mathbb{R}$, 
\begin{align} \label{11:14_2}
{\rm Var}[\chi (A_{I}(f_{\ell }; {\cal S}^2) )] =
\frac{\ell^3}{8 \pi }   {\cal I}^2+O(\ell^{5/2})
\end{align}
where 
$${\cal I}=\int_{I}  (-t^4+4 t^2 -1) e^{-\frac {t^2} 2 } d t.$$
\end{corollary}

\noindent Note that the asymptotic covariance in \eqref{11:14_1} can be
positive, negative or null depending on the choice of the intervals $I_{1}$
and $I_{2}$ (see Figure \ref{fig1}). From \eqref{11:14_1} and \eqref{11:14_2}
it follows also that, for every intervals $I_{1},I_{2}\subseteq \mathbb{R}$
such that the corresponding variances do not vanish, as $\ell $ goes to
infinity, $\chi (A_{I_{1}}(f_{\ell };\mathcal{S}^{2}))$ and $\chi
(A_{I_{2}}(f_{\ell };\mathcal{S}^{2}))$ are asymptotically perfectly
(positively or negatively) correlated, i.e. 
\begin{corollary} For all intervals $I_1,I_2$ such that 
$$
{\rm Var}[\chi (A_{I_1}(f_{\ell }; {\cal S}^2) )] , {\rm Var}[\chi (A_{I_2}(f_{\ell }; {\cal S}^2) )] \ne 0 
,$$
 as $\ell \to \infty$,
$$
 |{\rm Corr}[\chi (A_{I_1}(f_{\ell }; {\cal S}^2) ) \,,\, \chi (A_{I_2}(f_{\ell }; {\cal S}^2) ) ] |=1+O(\ell^{-1/2}).
$$ 
\end{corollary}

A similar form of degeneracy was earlier observed for level curves in
\noindent \cite{Wsurvey}. From Theorem \ref{cov} we also have  the
following corollary for
half-intervals $I_{1}=[u_{1},\infty )$ and $I_{2}=[u_{2},\infty )$ (see also Figure \ref{fig11} and Figure \ref{fig2}): 
\begin{corollary} \label{covcor} As $\ell \to \infty$, for $u_1, u_2 \in \mathbb{R}$,
\begin{align} \label{11:14_3}
{\rm Cov}[\chi (A_{u_1}(f_{\ell }; {\cal S}^2) ) \, ,\,  \chi (A_{u_2}(f_{\ell }; {\cal S}^2) ) ]=\frac{\ell^3}{8 \pi} 
 u_1   u_2  (u_1^2-1) (u_2^2-1)   e^{-\frac{u_2^2}{2}}   e^{-\frac{u_1^2}{2}} +O(\ell^{5/2}).
\end{align}
 In particular, if $u_1=u_2=u$, we can present an analytic expression for the variance:
\begin{align} \label{11:14_4}
{\rm Var}[\chi (A_{u}(f_{\ell }; {\cal S}^2) )]
&=\frac{\ell^3 }{8 \pi }  [H_3(u)+2H_1(u) ]^2 e^{-u^2}  +O(\ell^{5/2}),
\end{align}
where $H_q(\cdot)$ are the Hermite polynomial of order $q$.
\end{corollary}

\begin{figure}[h]
\centering
\subfigure[positive range]{\includegraphics[width=8cm,
height=7.5cm]{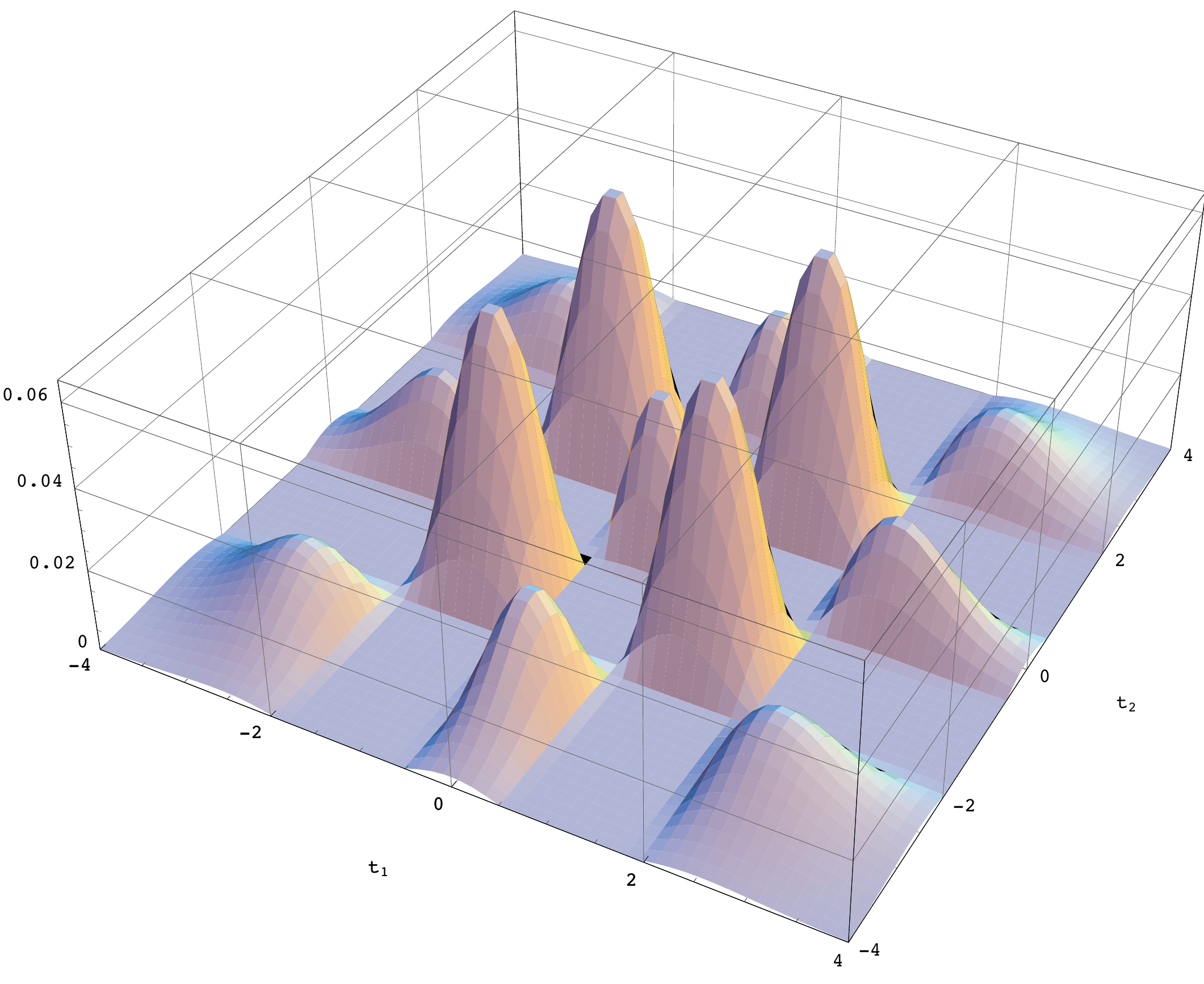} } 
\subfigure[negative
range]{\includegraphics[width=8cm, height=7.5cm]{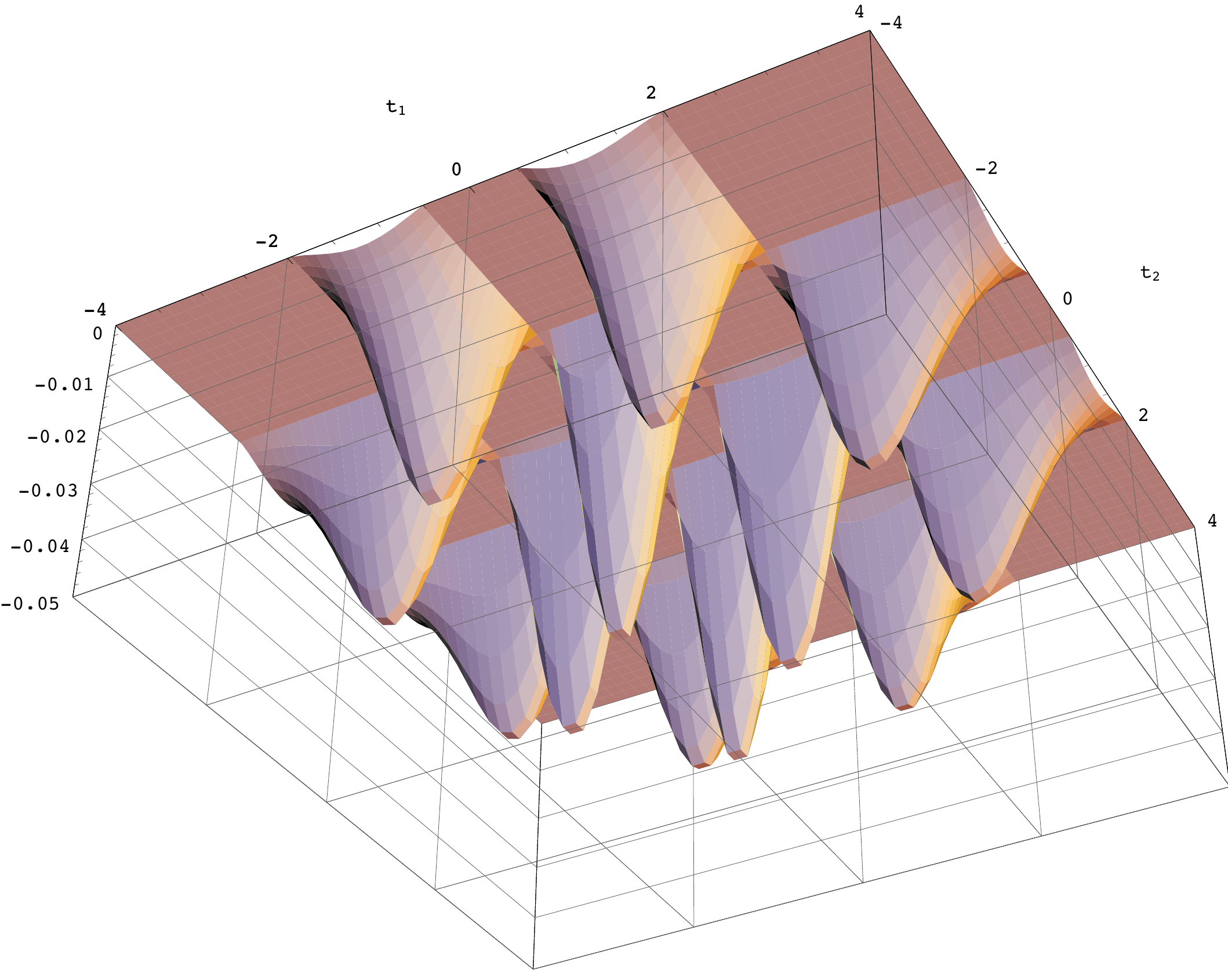} }
\caption{$\frac{1}{8 \protect\pi} (-t_1^4+4 t_1^2-1) (-t_2^4+4 t_2^2-1) e^{-%
\frac{t_1^2}{2}} e^{-\frac{t_2^2}{2}} $}
\label{fig1}
\end{figure}

\begin{figure}[h]
\centering
\subfigure[positive range]  {\includegraphics[width=8cm,
height=7.5cm]{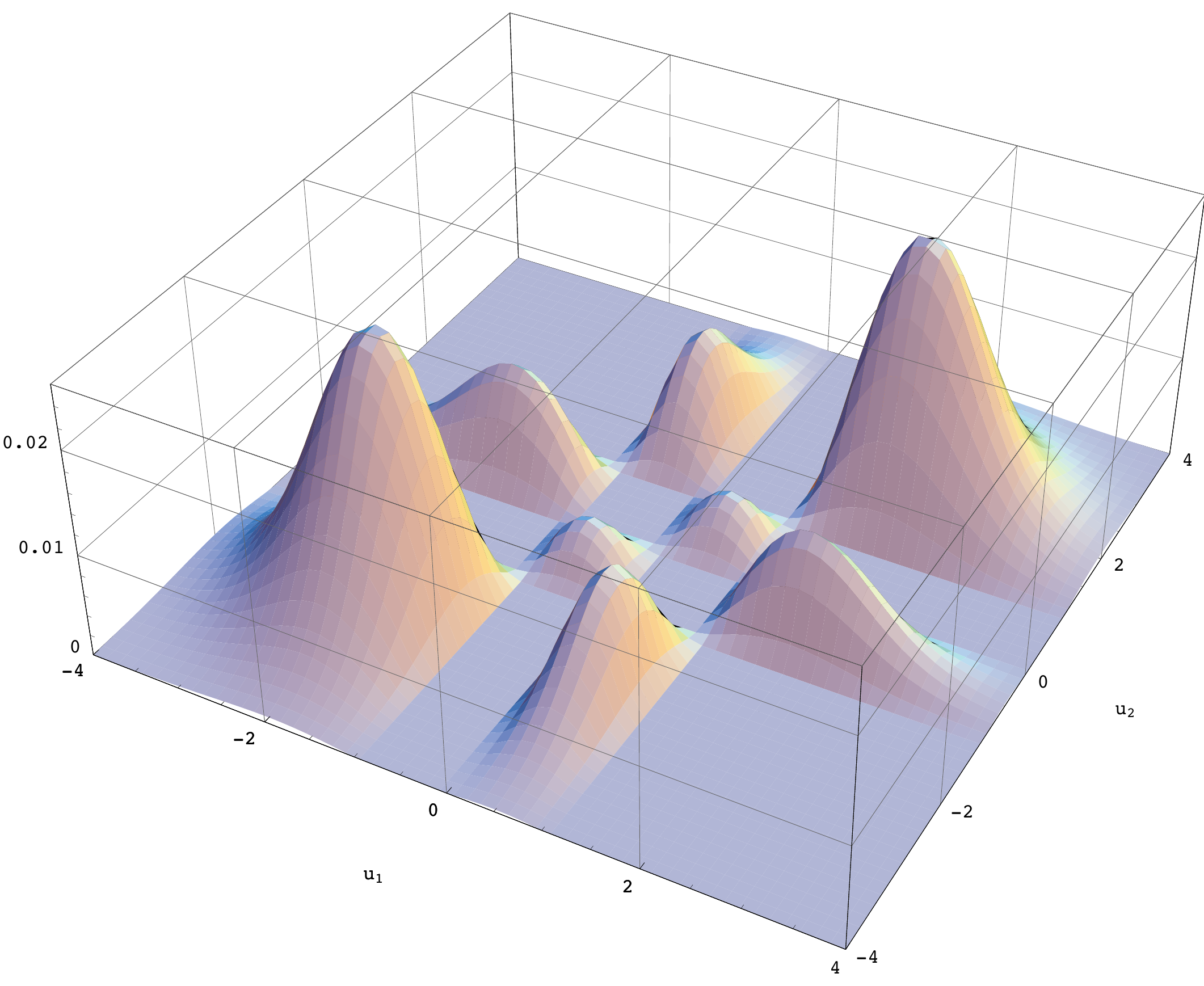} } 
\subfigure[negative
range]{\includegraphics[width=8cm, height=7.5cm]{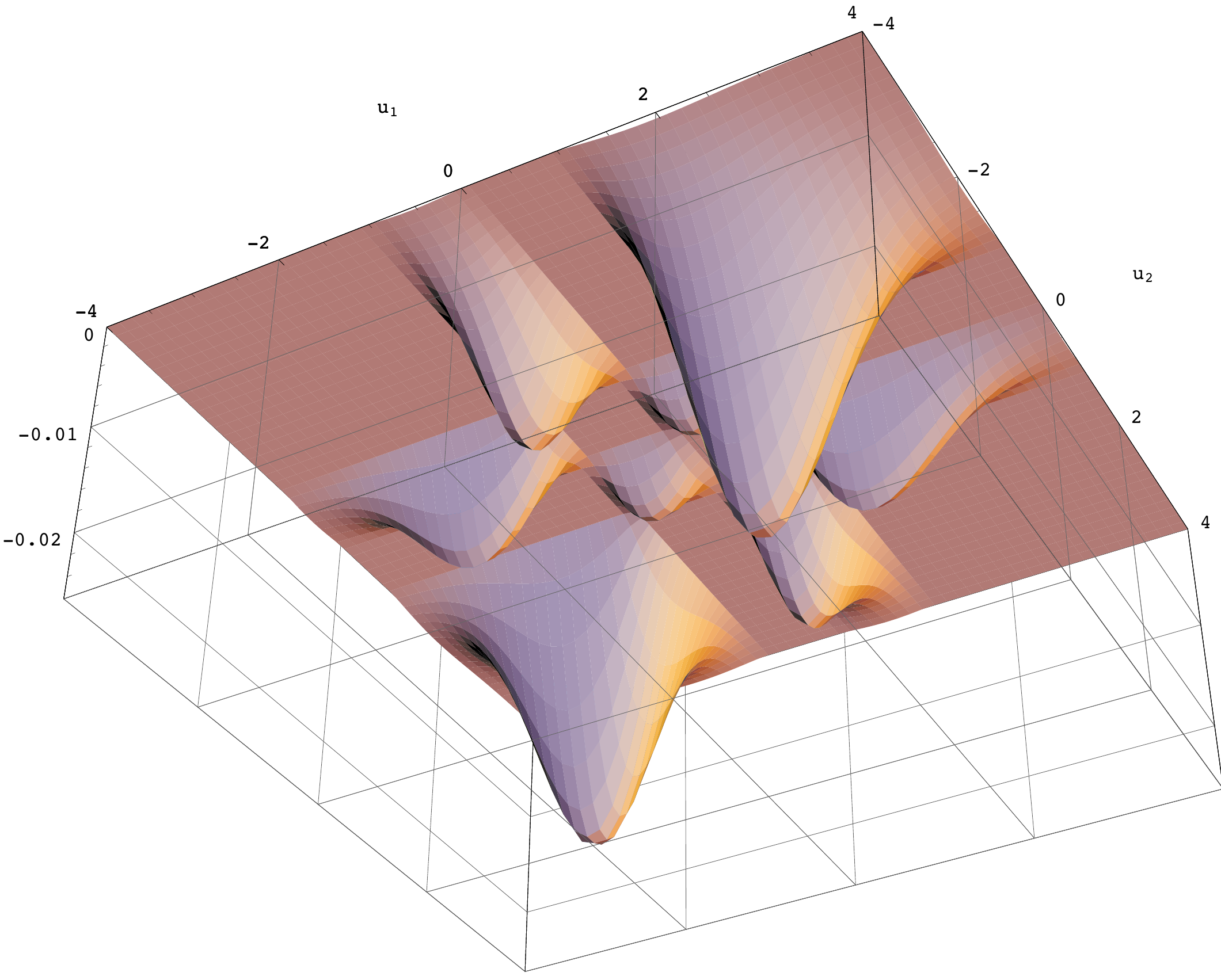} }
\caption{$\frac{1}{8 \protect\pi} u_1 u_2 (u_1^2-1) (u_2^2-1) e^{-\frac{u_1^2%
}{2}} e^{-\frac{u_2^2}{2}} $}
\label{fig11}
\end{figure}

\begin{figure}[h]
\centering {\includegraphics[width=8cm, height=5.8cm]{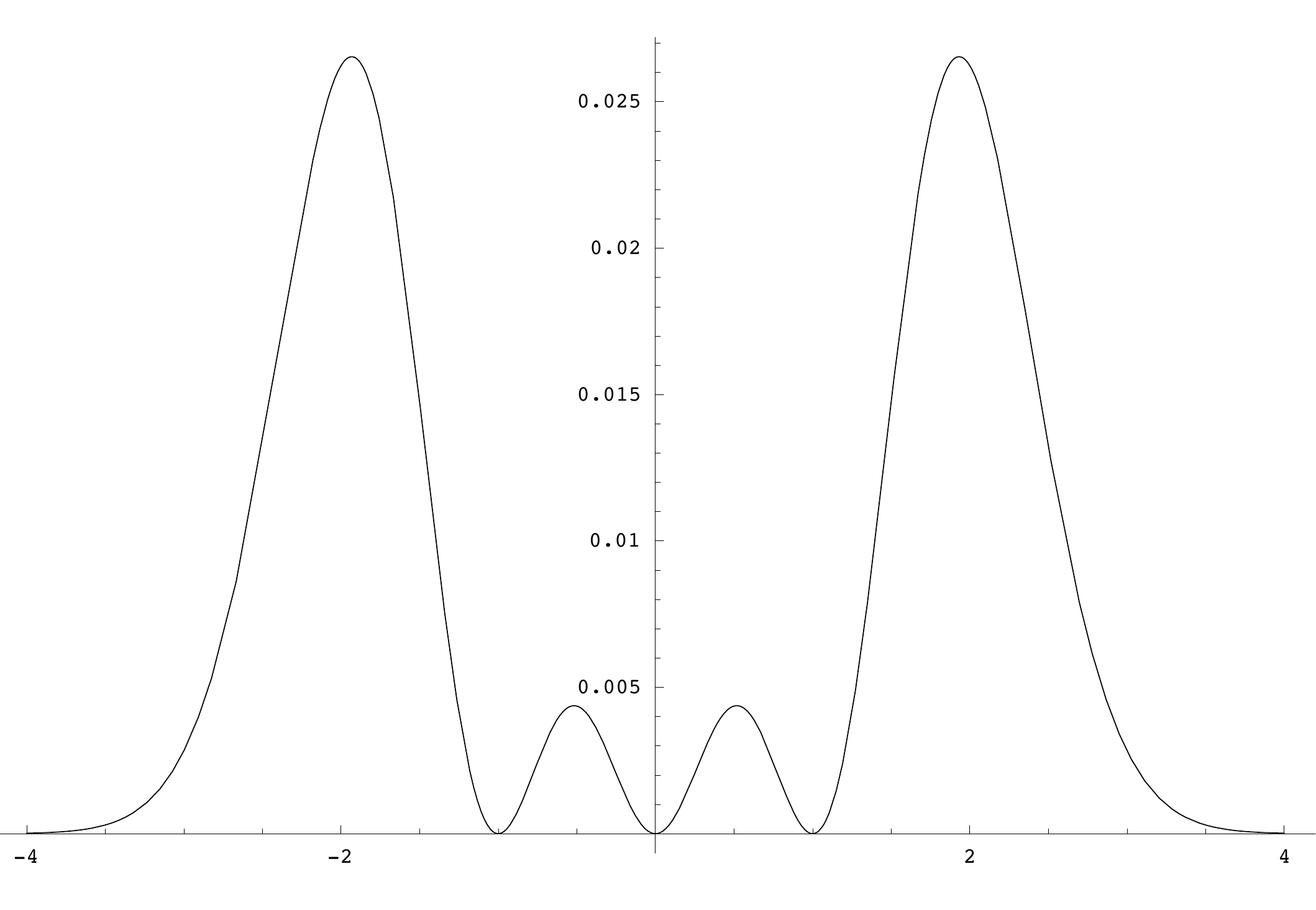}}
\caption{$\frac{1}{8 \protect\pi} [H_3(u)+2H_1(u) ]^2 e^{-u^2}$}
\label{fig2}
\end{figure}

As explained in Section \ref{proof}, the proof of Theorem \ref{cov} follows
from Morse theory and the analysis of asymptotic fluctuations of critical
points of random eigenfunctions \cite{CMW}. The expressions \eqref{11:14_1}-%
\eqref{11:14_4} are supported by extensive numerics \cite{CFMW}. 
\begin{remark}
Building upon previous results \cite{Wsurvey}, we are now able to present a full characterisation for the asymptotic behaviour for the variance of the  three Lipschitz-Killing curvatures ${\cal L}_i$, $i=0,1,2$, for the excursion sets of random spherical eigenfunctions on ${\cal S}^2$. In this setting these three LKC's correspond, respectively, to the EPC  (${\cal L}_0$), half the length of level curves (${\cal L}_1$), and the excursion area (${\cal L}_2$). Indeed, it was shown  \cite{DI} that the variance of the excursion area for spherical Gaussian eigenfunctions satisfies 
\begin{align} \label{11:21_1}
\lim_{\ell \to \infty} \ell \; {\rm Var}[{\cal L}_2(A_{u}(f_{\ell }; {\cal S}^2))]= u^2 \phi^2(u)=[ H_1(u) + H'_0(u)]^2 \phi(u)^2.
\end{align}
On the other hand \cite{Wsurvey} formula (18) (see also \cite{Wig}) asserts (in a slightly different form) that, for the variance of the boundary length of excursion sets, the following result holds
\begin{align} \label{11:21_2}
\lim_{\ell \to \infty} \ell^{-1} {\rm Var}[{\cal L}_1(A_{u}(f_{\ell }; {\cal S}^2))] = {\rm const} \times   u^4 \phi^2(u)={\rm const} \times [H_2(u)+H'_1(u) ]^2 \phi^2(u).
\end{align}
Likewise the asymptotic variance of the Euler-Poincar\'e characteristic, derived in Corollary \ref{covcor}, may be written as 
\begin{align} \label{11:21_3}
\lim_{\ell \to \infty} \ell^{-3} {\rm Var}[{\cal L}_0(A_{u}(f_{\ell }; {\cal S}^2))]=\frac{1}{4 } [u^3-u]^2 \phi^2(u)= \frac{1}{4 }  [H_3(u)+H'_2(u) ]^2 \phi^2(u).
\end{align}
We may unite the asymptotic expressions for the variance of the first three  Lipschitz-Killing curvatures in \eqref{11:21_1}, \eqref{11:21_2} and \eqref{11:21_3} into a single formula: 
\begin{align} \label{11:25}
 \lim_{\ell \to \infty} \ell^{2k-3} \times {\rm Var}[{\cal L}_k(A_u(f_{\ell }; {\cal S}^2) ] = const \times [H_{3-2k}(u)+ H'_{2-2k}(u)]^2 \; \phi^2(u), \hspace{0.5cm}  k=0,1,2.
 \end{align}
 A comparison of \eqref{11:25} with expressions  \eqref{GKF2}, \eqref{GKF1} and \eqref{expect} seems to suggest the existence of an (asymptotic) second order Gaussian Kinematic Formula for spherical Gaussian eigenfunctions. We leave  the investigation of the general validity of such en expression for higher dimensional spheres to future research. 
 \end{remark} 
\begin{remark}
 For all $u \in \mathbb{R}$, we have 
$$\frac{\chi(A_{u}(f_{\ell }; {\cal S}^2))}{\mathbb{E}[\chi(A_{u}(f_{\ell }; {\cal S}^2))]}-1=O_p\Big(	\frac{1}{\sqrt{\ell}} \Big),$$
with the usual convention  $X_n=O_p(a_n)$ meaning that the sequence $|X_n|/a_n$ is bounded in probability;
 i.e., in the high frequency limit $\ell \rightarrow \infty$, the ratio of the realised and expected value for the  EPC of the excursion will converge to unity in probability for all $u \in \mathbb{R}$.

 \end{remark} 

\section{Background on Morse theorem and (approximate) Kac-Rice formula}

\subsection{Morse theorem}

We start by recalling a general expression for the EPC by means of so-called
Morse Theorem (see \cite{adlertaylor} Section 9.3). Assuming that $\mathcal{M%
}$ is a $C^2$ manifold without boundary in $\mathbb{R}^N$ and that $h \in
C^2(\mathcal{M})$ is a Morse function on $\mathcal{M}$ (i.e. its Hessian is
non degenerate at the critical points), we have 
\begin{align}  \label{epc_morse}
\chi({\mathcal{M}})=\sum_{j=0}^{\text{dim}(\mathcal{M})} (-1)^j \mu_j({%
\mathcal{M}},h),
\end{align}
where $\mu_j(\mathcal{M},h)$ is the number of critical points of $h$ with
Morse index $j$, i.e., the Hessian of $h$ has $j$ negative eigenvalues. In
order to develop our results we will need to exploit \eqref{epc_morse} in
the case of excursion sets of spherical eigenfunctions; to this end, let us
first recall some basic differential geometry on $\mathcal{S}^2$. The metric
tensor on the tangent plane $T(\mathcal{S}^2)$ is given by 
\begin{equation*}
g(\theta, \varphi )=\left [ 
\begin{matrix}
1 & 0 \\ 
0 & \sin^2 \theta%
\end{matrix}
\right].
\end{equation*}
For $x=(\theta, \varphi) \in \mathcal{S}^2 \setminus \{N,S\}$ ($N,S$ are the
north and south poles i.e. $\theta=0$ and $\theta=\pi$ respectively), the
vectors 
\begin{equation*}
\vec{e}_{\theta}=\frac{\partial}{\partial \theta}, \hspace{2cm} \vec{e}%
_{\varphi}=\frac{1}{\sin \theta} \frac{\partial}{\partial \varphi},
\end{equation*}
constitute an orthonormal basis for $T_x(\mathcal{S}^2)$; in these
coordinates the gradient is given by $\nabla=(\frac{\partial}{\partial \theta%
}, \frac{1}{\sin \theta} \frac{\partial}{\partial \varphi})$. The Hessian of
a function $f \in C^2(\mathcal{S}^2)$ is the bilinear symmetric map from $%
C^1(T(\mathcal{S}^2)) \times C^1(T(\mathcal{S}^2))$ to $C^0(\mathcal{S}^2)$
defined by 
\begin{equation*}
\nabla^2 f(X,Y)=X Y f - \nabla_X Y f, \hspace{1cm} X,Y \in T(\mathcal{S}^2),
\end{equation*}
where $\nabla_X$ denotes Levi-Civita connection (see e.g. \cite{adlertaylor}
Chapter 7 for more discussion and details). For our computations to follow
we shall need the matrix-valued process $\nabla^2_E {f_{\ell}}(x)$ with
elements given by 
\begin{equation*}
\{\nabla^2_E {f_{\ell}}(x)\}_{a,b=\theta,\varphi}=\{(\nabla^2{f_{\ell}}(x))(%
\vec{e}_{a},\vec{e}_b)\}_{a,b=\theta,\varphi},
\end{equation*}
where $E=\{\vec{e}_{\theta}, \vec{e}_{\varphi}\}$. In coordinates as above,
this matrix can be expressed as 
\begin{align*}
\nabla^2_E f_{\ell}(x)&=\left[ 
\begin{matrix}
\frac{\partial^2}{\partial \theta^2}-\Gamma^{\theta}_{\theta \theta} \frac{%
\partial}{\partial \theta} - \Gamma^{\varphi}_{\theta \theta} \frac{\partial%
}{\partial \varphi} & \frac{1}{\sin \theta} [\frac{\partial^2}{\partial
\theta \partial \varphi }-\Gamma^{\varphi}_{ \varphi \theta} \frac{\partial}{%
\partial \varphi} - \Gamma^{\theta}_{\theta \varphi} \frac{\partial}{%
\partial \theta}] \\ 
\frac{1}{\sin\theta} [\frac{\partial^2}{\partial \theta \partial \varphi }%
-\Gamma^{\varphi}_{ \varphi \theta} \frac{\partial}{\partial \varphi} -
\Gamma^{\theta}_{\theta \varphi} \frac{\partial}{\partial \theta} ] & \frac{1%
}{\sin^2 \theta}[\frac{\partial^2}{\partial \varphi^2}-\Gamma^{\varphi}_{%
\varphi \varphi} \frac{\partial}{\partial \varphi} -
\Gamma^{\theta}_{\varphi \varphi } \frac{\partial}{\partial \theta} ]%
\end{matrix}
\right] \\
&=\left[ 
\begin{matrix}
\frac{\partial^2}{\partial \theta^2} & \frac{1}{\sin \theta}[\frac{\partial^2%
}{\partial \theta \partial \varphi}-\frac{\cos \theta}{\sin \theta} \frac{%
\partial}{\partial \varphi}] \\ 
\frac{1}{\sin \theta}[\frac{\partial^2}{\partial \theta \partial \varphi}-%
\frac{\cos \theta}{\sin \theta} \frac{\partial}{\partial \varphi} ] & \frac{1%
}{\sin^2 \theta}[\frac{\partial^2}{\partial \varphi^2}+\sin \theta \cos
\theta \frac{\partial}{\partial \theta}]%
\end{matrix}
\right].
\end{align*}
Here $\Gamma_{ab}^c$ are the usual Christoffel symbols, see e.g. \cite%
{chavel} Section I.1, which allow to compute the Levi-Civita connection: 
\begin{equation*}
\nabla_{\vec{e}_a}\vec{e}_b=\Gamma_{a b}^{\theta} \vec{e}_{\theta}+\Gamma_{a
b}^{\varphi} \vec{e}_{\varphi}, \hspace{0.7cm} a,b=\theta, \varphi.
\end{equation*}
More explicitly, Christoffel symbols for $\mathcal{S}^2$ are given by 
\begin{equation*}
\Gamma^{\theta}_{ \theta \varphi}= \Gamma^{\theta}_{\theta
\theta}=\Gamma^{\varphi}_{\varphi \varphi}=\Gamma^{\varphi}_{\theta
\theta}=0, \hspace{0.5cm}\Gamma_{\varphi \varphi}^{\theta}=-\sin \theta \cos
\theta, \hspace{0.5cm}\Gamma_{ \varphi \theta }^{\varphi}=\cot \theta.
\end{equation*}

We now state the Morse representation for the Euler characteristic of the
excursion set: let $\mathcal{M}$ and $h$ in \eqref{epc_morse} be $%
A_I(f_{\ell};\mathcal{S}^2)$ and $\left. f_\ell \right|_{{A_I(f_{\ell};%
\mathcal{S}^2)}}$ respectively, we have 
\begin{equation}  \label{morse}
\chi (A_{I}(f_{\ell }; {\mathcal{S}}^2) )=\sum_{j=0}^{2}(-1)^{j} \mu_j,
\end{equation}
where 
\begin{equation*}
\mu_j=\# \{x \in {\mathcal{S}}^2: f_{\ell }(x)\in I, \nabla f_{\ell }(x)=0, 
\text{Ind} (-\nabla^2_E f_{\ell }(x)) =j \},
\end{equation*}
$\text{Ind}(M)$ denoting the number of negative eigenvalues of a square
matrix $M$. More specifically, $\mu _{0}$ is the number of maxima, $\mu_{1}$
the number of saddles, and $\mu _{2}$ the number of minima in the excursion
region $A_{I}(f_{\ell }; \mathcal{S}^2)$.

\subsection{Kac-Rice formula}

The Kac-Rice formula is a standard tool (or meta-theorem) for expressing the
(factorial) moments of the zero crossings number of a Gaussian process in
terms of certain explicit integrals. In our case, we are interested in
counting the critical points of $f_{\ell }$, i.e. the zeros of the map $%
x\rightarrow \nabla f_{\ell }(x)$. Let $\mathcal{E}\subset \mathbb{R}^{n}$
be a nice Euclidean domain, and $g:\mathcal{E}\rightarrow \mathbb{R}^{n}$ a
centred Gaussian random field, a.s. smooth. Define the $2$-point correlation
function of critical points $K_{2}=K_{2;g}:\mathcal{E}^{2}\rightarrow 
\mathbb{R}$ 
\begin{equation*}
K_{2}(x,y)=\phi _{(\nabla g(x),\nabla g(y))}(\mathbf{0},\mathbf{0})\cdot 
\mathbb{E}[|\mathrm{det}\nabla ^{2}g(x)|\cdot |\mathrm{det}\nabla ^{2}g(y)|%
\big|\nabla g(x)=\nabla g(y)=\mathbf{0}],
\end{equation*}%
where $\phi _{(\nabla g(x),\nabla g(y))}$ is the Gaussian probability
density of $(\nabla g(x),\nabla g(y))\in \mathbb{R}^{2n}$. Let $\mathcal{N}%
^{c}(g)=\mathcal{N}^{c}(g,\mathcal{E})=\#\{x\in \mathcal{E}:\nabla g(x)=%
\mathbf{0}\}$; by virtue of \cite{azaiswschebor} Theorem 6.3, we have 
\begin{equation*}
\mathbb{E}[{\mathcal{N}}^{c}(g,{\mathcal{E}})\cdot ({\mathcal{N}}^{c}(g,{%
\mathcal{E}})-1)]=\iint_{\mathcal{E}\times \mathcal{E}}K_{2}(x,y)dxdy,
\end{equation*}%
provided that the Gaussian distribution of $\left( \nabla g(x),\nabla
g(y)\right) \in \mathbb{R}^{2n}$ is non-degenerate for all $(x,y)\in 
\mathcal{E}^{2}$, on the validity condition of Kac-Rice formula in the
Gaussian case, see \cite{azaiswschebor} Theorem 6.3 and Proposition 1.2, and 
\cite{rudnickwigmanyesha} Section 1.4. Moreover for $\mathcal{D}_{1},%
\mathcal{D}_{2}\subseteq \mathcal{E}$ two nice \textit{disjoint} domains,
under the same non-degeneracy assumptions for every $(x,y)\in \mathcal{D}%
_{1}\times \mathcal{D}_{2}$, we have 
\begin{equation*}
\mathbb{E}[{\mathcal{N}}^{c}(g,{\mathcal{D}}_{1})\cdot {\mathcal{N}}^{c}(g,{%
\mathcal{D}}_{2})]=\iint_{\mathcal{D}_{1}\times \mathcal{D}%
_{2}}K_{2}(x,y)dxdy.
\end{equation*}%
It is easy to adapt the definition of the $2$-point correlation function in 
order to investigate, for example, the maxima with values lying in an interval $%
I\subseteq \mathbb{R}$:   we re-define $K_{2}$ as 
\begin{align*}
& K_{2,0,0}(x,y;I,I)=\phi _{(\nabla g(x),\nabla g(y))}(\mathbf{0},\mathbf{0})
\\
& \hspace{-0.3cm}\times \mathbb{E}[|\mathrm{det}\nabla ^{2}g(x)|\cdot |%
\mathrm{det}\nabla ^{2}g(y)|\cdot 1\hspace{-0.27em}\mbox{\rm l}%
_{I}(g(x))\cdot 1\hspace{-0.27em}\mbox{\rm l}_{I}(g(y))\cdot 1\hspace{-0.27em%
}\mbox{\rm l}_{\{\text{Ind}(-\nabla ^{2}g(x))=0\}}\cdot 1\hspace{-0.27em}%
\mbox{\rm l}_{\{\text{Ind}(-\nabla ^{2}g(y))=0\}}\big|\nabla g(x)=\nabla
g(y)=\mathbf{0}],
\end{align*}%
where $1\hspace{-0.27em}\mbox{\rm l}_{I}$ is the characteristic function of $%
I$ on $\mathbb{R}$.

For the Kac-Rice formula on manifolds we refer to \cite{adlertaylor} Theorem
12.1.1, in particular, let 
\begin{equation*}
\mathcal{N}_{I}^{c}(f_{\ell })=\#\{x\in \mathcal{S}^{2}:\;f_{\ell }(x)\in
I,\nabla f_{\ell }(x)=0\}
\end{equation*}%
be the total number of critical points in $I$ of $\{f_{\ell }(x),x\in 
\mathcal{S}^{2}\}$; we have 
\begin{equation} \label{eq:Kac-Rice2ndmomgen}
\mathbb{E}[\mathcal{N}_{I}^{c}(f_{\ell })\cdot (\mathcal{N}_{I}^{c}(f_{\ell
})-1)]=\iint\limits_{\mathcal{S}^{2}\times \mathcal{S}^{2}}K_{2,\ell
}(x,y;I,I)dxdy,
\end{equation}
where 
\begin{align}
K_{2,\ell }(x,y;I,I)& =\phi _{(\nabla f_{\ell }(x),\nabla f_{\ell }(y))}(%
\mathbf{0},\mathbf{0})  \notag  \label{K2} \\
& \;\;\times \mathbb{E}[|\det \nabla _{E}^{2}f_{\ell }(x)|\cdot |\det \nabla
_{E}^{2}f_{\ell }(y)|\cdot 1\hspace{-0.27em}\mbox{\rm l}_{I}(f_{\ell
}(x))\cdot 1\hspace{-0.27em}\mbox{\rm l}_{I}(f_{\ell }(y))\big|\nabla
f_{\ell }(x)=\nabla f_{\ell }(y)=\mathbf{0}].
\end{align}%
One technical difficulty in working with the spherical Gaussian
eigenfunctions $f_{\ell }$ in \eqref{eq:Kac-Rice2ndmomgen} is related to the
fact that the Gaussian distribution of $(f_{\ell }(x),\nabla f_{\ell
}(x),\nabla _{E}^{2}f_{\ell }(x))$ is always degenerate. However, this issue
can be handled by writing $f_{\ell }$ as a linear combination of second
order derivatives, and thus reducing the dimension of the Gaussian vector
involved in the evaluation of $K_{2}$, see \cite{CMW}.

A much trickier issue arises when we need to validate a sufficient
non-degeneracy assumptions due to the technical difficulties of dealing with 
$10\times 10$ matrices depending on both $x$ and $y$ (and $\ell $).
Following \cite{CMW} and \cite{rudnickwigman}, we do not claim the (precise)
Kac-Rice formula \eqref{eq:Kac-Rice2ndmomgen} but rather an approximate
version, see \cite{CMW} formula (3.5), equivalent to %
\eqref{eq:Kac-Rice2ndmomgen} up to an admissible error.

First note that, by isotropy, $K_{2}(x,y)=K_{2}(d(x,y))$ depends only on the
(spherical) distance $d(x,y)=\arccos (\left\langle x,y\right\rangle )$
between $x$ and $y$. In view of this, we note that it is convenient to
perform our computations along a specific geodesic; in particular, we
constrain ourselves to the equatorial line $\theta_x=\theta_y=\pi/2$; it is
immediate to see that here the gradient and the Hessian are 
\begin{equation*}
\left.\nabla \right|_{\theta= \pi /2}=(\frac{\partial}{\partial \theta}, 
\frac{\partial}{\partial \varphi}), \hspace{1cm} \left. \nabla^2_E
\right|_{\theta= \pi /2}=\left[ 
\begin{matrix}
\frac{\partial^2}{\partial \theta^2} & \frac{\partial^2}{\partial \theta
\partial \varphi } \\ 
\frac{\partial^2}{\partial \theta \partial\varphi } & \frac{\partial^2}{%
\partial \varphi^2}%
\end{matrix}
\right].
\end{equation*}

The basic idea is to split the range of integration in %
\eqref{eq:Kac-Rice2ndmomgen} into two parts: the ``short range'' regime $%
d(x,y)<C/\ell $ and the ``long range'' regime $d(x,y)>C/\ell$, $C$ denoting
a sufficiently big positive constant. In the short range regime Kac-Rice
formula holds only approximately, but, by a partitioning argument inspired
from \cite{rudnickwigman} (see also \cite{Wig}), it is possible to prove
that its contribution is $O(\ell^{2})$. In the long range regime $%
d(x,y)>C/\ell$ the Kac-Rice formula is precise. The above yields 
\begin{align}  \label{A_K-R}
\mathbb{E }[\mathcal{N}_{I}^{c}(f_{\ell})\cdot (\mathcal{N}%
_{I}^{c}(f_{\ell})-1)] = \int_{d(x,y)>C/\ell} \iint_{I \times I} K_{2,\ell}
(x,y; t_1,t_2) d t_1 d t_2 d x d y+O(\ell^{2}),
\end{align}
where 
\begin{align}  \label{4:29}
&K_{2,\ell} (x,y; t_1,t_2) \\
&\;\;= \varphi_{x,y,\ell}(t_1,t_2,\mathbf{0},\mathbf{0}) \mathbb{E }[ |\det
\nabla_E^2 f_{\ell} (x)|\cdot |\det \nabla^2_E f_{\ell} (y)| \big| \nabla
f_{\ell}(x)=\nabla f_{\ell}(y)=\mathbf{0}, f_{\ell}(x)=t_1, f_{\ell}(y)=t_2],
\notag
\end{align}
and $\varphi_{x,y,\ell}$ is the density of the $6$-dimensional vector $%
(f_{\ell}(x),f_{\ell}(y),\nabla f_{\ell}(x), \nabla f_{\ell}(y))$. For
further details on the proof of \eqref{A_K-R}, see \cite{CMW} Section 3.4.1
and Section 3.4.2.

As it will become clear from the proof of Proposition \ref{lemma3} below, we
obtain a considerable simplification in our calculations since, during the
application of the (approximate) Kac-Rice formula for studying the variance
of the EPC, we can get rid of the absolute values in \eqref{4:29}; in fact,
for $g$ as before a smooth, centred Gaussian random field, we observe that (%
\cite{adlerstflour} Lemma 4.2.2) 
\begin{align*}
(-1)^{j}\;|\det \nabla _{E}^{2}g(x)|1\hspace{-0.27em}\mbox{\rm l}_{\{\text{%
Ind}(-\nabla _{E}^{2}g(x))=j\}}& =(-1)^{j}\;\text{sgn}(\det \nabla
_{E}^{2}g(x))(\det \nabla _{E}^{2}g(x))1\hspace{-0.27em}\mbox{\rm l}_{\{%
\text{Ind}(-\nabla _{E}^{2}g(x))=j\}} \\
& =\det (-\nabla _{E}^{2}g(x))1\hspace{-0.27em}\mbox{\rm l}_{\{\text{Ind}%
(-\nabla _{E}^{2}g(x))=j\}},
\end{align*}%
since $\text{sgn}(\det \nabla _{E}^{2}g(x))1\hspace{-0.27em}\mbox{\rm l}_{\{%
\text{Ind}(-\nabla _{E}^{2}g(x))=j\}}=(-1)^{j+1}$ and $-\det \nabla
_{E}^{2}g(x)=\det (-\nabla _{E}^{2}g(x))$. Hence 
\begin{equation} \label{simp}
\sum_{j=0}^{2}(-1)^{j}\;|\det \nabla _{E}^{2}g(x)|1\hspace{-0.27em}%
\mbox{\rm
l}_{\{\text{Ind}(-\nabla _{E}^{2}g(x))=j\}}=\det (-\nabla _{E}^{2}g(x)).
\end{equation}

\section{Proof of Theorem \protect\ref{cov}}

\label{proof} Let $I_i \subseteq \mathbb{R}$, $i=1,2$ be two interval in the
real line; in the argument to follow we shall adopt the following notation: 
\begin{align*}
\mu_{j;i}(f_{\ell })=\# \{x \in {\mathcal{S}}^2: f_{\ell }(x)\in I_i, \nabla
f_{\ell }(x)=0, \text{Ind}( - \nabla_E^2f_{\ell } (x) ) =j\}, \hspace{1cm}
j=0,1,2,\; i=1,2.
\end{align*}
Theorem \ref{cov} is a straightforward application of Proposition \ref%
{lemma2} and Proposition \ref{lemma3}. The first building block is the
approximate Kac-Rice formula for covariance computation: 
\begin{proposition} \label{lemma2} There exists a constant $C > 0$ sufficiently big, such that
\begin{align} \label{sim}
\sum_{j,k=0}^{2} (-1)^{j+k} \E[\mu_{j;1}(f_{\ell }) \mu_{k;2}(f_{\ell })]= \int_{d(x,y)>C/\ell} \iint_{I_1 \times I_2} {J}_{2,\ell} (x,y; t_1,t_2) d t_1 d t_2 d x d y+O(\ell^{2})
\end{align}
where
\begin{align} \label{J2}
 &{J}_{2,\ell} (x,y; t_1,t_2) \\
 &\;\;= \varphi_{x,y,\ell}(t_1,t_2,{\bf 0},{\bf 0}) \; \E [ \det(- \nabla_E^2 f_{\ell} (x))\cdot \det(- \nabla_E^2 f_{\ell} (y)) \big|
  \nabla f_{\ell}(x)=\nabla f_{\ell}(y)={\bf 0}, f_{\ell}(x)=t_1, f_{\ell}(y)=t_2]. \nonumber
\end{align}
\end{proposition}

\begin{proof}
We start by observing that 
\begin{equation} \label{5:36}
\sum_{j,k=0}^{2}(-1)^{j+k}\mathbb{E}[\mu _{j;1}(f_{\ell })\mu _{k;2}(f_{\ell
})]=\sum_{j\neq k}(-1)^{j+k}\mathbb{E}[\mu _{j;1}(f_{\ell })\mu
_{k;2}(f_{\ell })]+\sum_{j=0}^{2}\mathbb{E}[\mu _{j;1}(f_{\ell })\mu
_{j;2}(f_{\ell })].
\end{equation}
For the non diagonal terms in \eqref{5:36} with $j\neq k$, $j,k=0,1,2$, 
we directly obtain (see \cite{CMW} Section 3.4) that for any
sufficiently big constant $C>0$, we have 
\begin{equation*}
\mathbb{E}[\mu _{j;1}(f_{\ell })\mu _{k;2}(f_{\ell })]=\int_{d(x,y)>C/\ell
}\iint_{I_{1}\times I_{2}}\tilde{K}_{2,\ell
,j,k}(x,y;t_{1},t_{2})dt_{1}dt_{2}dxdy+O(\ell ^{2}),
\end{equation*}%
where 
\begin{align*}
\tilde{K}_{2,\ell ,j,k}(x,y;t_{1},t_{2})& =\varphi _{x,y,\ell }(t_{1},t_{2},%
\mathbf{0},\mathbf{0}) \\
& \hspace{-3.5cm}\times \mathbb{E}[|\det \nabla _{E}^{2}f_{\ell }(x)|\cdot
|\det \nabla _{E}^{2}f_{\ell }(y)|\cdot 1\hspace{-0.27em}\mbox{\rm l}_{\{%
\mathrm{Ind}(-\nabla _{E}^{2}f_{\ell }(x))=j\}}\cdot 1\hspace{-0.27em}%
\mbox{\rm l}_{\{\mathrm{Ind}(-\nabla _{E}^{2}f_{\ell }(y))=k\}}\big|\nabla
f_{\ell }(x)=\nabla f_{\ell }(y)=\mathbf{0},f_{\ell }(x)=t_{1},f_{\ell
}(y)=t_{2}].
\end{align*}%
To work out the diagonal terms $\mathbb{E}[\mu _{j;1}(f_{\ell })\mu
_{j;2}(f_{\ell })]$, $j=0,1,2$, in \eqref{5:36}, we introduce the following
notation: 
\begin{align*}
\mu _{j;i\setminus i^{\prime }}(f_{\ell })& =\#\{x\in {\mathcal{S}}%
^{2}:f_{\ell }(x)\in I_{i}\setminus I_{i^{\prime }},\nabla f_{\ell }(x)=0,%
\text{Ind}(-\nabla _{E}^{2}f_{\ell }(x))=j\}, \\
\mu _{j;i\cap i^{\prime }}(f_{\ell })& =\#\{x\in {\mathcal{S}}^{2}:f_{\ell
}(x)\in I_{i}\cap I_{i^{\prime }},\nabla f_{\ell }(x)=0,\text{Ind}(-\nabla
_{E}^{2}f_{\ell }(x))=j\},
\end{align*}%
with $i,i^{\prime }=1,2$, so that 
\begin{equation*}
\mu _{j;1}(f_{\ell })\mu _{j;2}(f_{\ell })=(\mu _{j;1\setminus 2}(f_{\ell
})+\mu _{j;1\cap 2}(f_{\ell }))(\mu _{j;2\setminus 1}(f_{\ell })+\mu
_{j;1\cap 2}(f_{\ell }))
\end{equation*}%
and then 
\begin{align}
& \mathbb{E}[\mu _{j;1}(f_{\ell })\mu _{j;2}(f_{\ell })]  \notag \\
& =\mathbb{E}[\mu _{j;1\setminus 2}(f_{\ell })\mu _{j;2\setminus 1}(f_{\ell
})]+\mathbb{E}[\mu _{j;1\setminus 2}(f_{\ell })\mu _{j;1\cap 2}(f_{\ell })]+%
\mathbb{E}[\mu _{j;2\setminus 1}(f_{\ell })\mu _{j;1\cap 2}(f_{\ell })]+%
\mathbb{E}[\mu _{j;1\cap 2}(f_{\ell })\mu _{j;2\cap 1}(f_{\ell })]  \notag \\
& =\mathbb{E}[\mu _{j;1\setminus 2}(f_{\ell })\mu _{j;2\setminus 1}(f_{\ell
})]+\mathbb{E}[\mu _{j;1\setminus 2}(f_{\ell })\mu _{j;1\cap 2}(f_{\ell })]+%
\mathbb{E}[\mu _{j;2\setminus 1}(f_{\ell })\mu _{j;1\cap 2}(f_{\ell })] 
\notag \\
& \;\;+\mathbb{E}[\mu _{j;1\cap 2}(f_{\ell })(\mu _{j;1\cap 2}(f_{\ell
})-1)]+\mathbb{E}[\mu _{j;1\cap 2}(f_{\ell })].  \label{pezzi}
\end{align}%
For the last term in \eqref{pezzi} we note that the expected value of the
EPC of the excursion set is $O(\ell ^{2})$, while for the other terms we can
apply again the approximate Kac-Rice formula. For example we have: 
\begin{align*}
\mathbb{E}[\mu _{j;1\setminus 2}(f_{\ell })\mu _{j;2\setminus 1}(f_{\ell
})]& =\int_{d(x,y)>C/\ell }\iint_{I_{1}\setminus I_{2}\times I_{2}\setminus
I_{1}}\tilde{K}_{2,\ell ,j,j}(x,y;t_{1},t_{2})dt_{1}dt_{2}dxdy+O(\ell ^{2}),
\\
\mathbb{E}[\mu _{j;1\cap 2}(f_{\ell })(\mu _{j;1\cap 2}(f_{\ell })-1)]&
=\int_{d(x,y)>C/\ell }\iint_{I_{1}\cap I_{2}\times I_{1}\cap I_{2}}\tilde{K}%
_{2,\ell ,j,j}(x,y;t_{1},t_{2})dt_{1}dt_{2}dxdy+O(\ell ^{2}),
\end{align*}%
and then 
\begin{equation*}
\mathbb{E}[\mu _{j;1}(f_{\ell })\mu _{j;2}(f_{\ell })]=\int_{d(x,y)>C/\ell
}\iint_{I_{1}\times I_{2}}\tilde{K}_{2,\ell
,j,j}(x,y;t_{1},t_{2})dt_{1}dt_{2}dxdy+O(\ell ^{2}).
\end{equation*}%
We can apply now the identity \eqref{simp} to get: 
\begin{equation*}
\sum_{j,k=0}^{2}(-1)^{j+k}\mathbb{E}[\mu _{j;1}(f_{\ell })\mu _{k;2}(f_{\ell
})]=\int_{d(x,y)>C/\ell }\iint_{I_{1}\times I_{2}}{J}_{2,\ell
}(x,y;t_{1},t_{2})dt_{1}dt_{2}dxdy+O(\ell ^{2})
\end{equation*}%
where 
\begin{align*}
& J_{2,\ell }(x,y;t_{1},t_{2}) \\
& \;\;=\varphi _{x,y,\ell }(t_{1},t_{2},\mathbf{0},\mathbf{0})\;\mathbb{E}%
[\det (-\nabla _{E}^{2}f_{\ell }(x))\cdot \det (-\nabla _{E}^{2}f_{\ell }(y))%
\big|\nabla f_{\ell }(x)=\nabla f_{\ell }(y)=\mathbf{0},f_{\ell
}(x)=t_{1},f_{\ell }(y)=t_{2}].
\end{align*}
\end{proof}

Our second tool yields an analytic expression for the alternating sum in the
variance computation. 
\begin{proposition}
\label{lemma3}
\begin{align*}
&\sum_{j,k=0}^{2} (-1)^{j+k} \E[\mu_{j;1}(f_{\ell }) \mu_{k;2} (f_{\ell }) ] - \E[\chi (A_{I_1}(f_{\ell}; {\cal S}^2))] \E[\chi (A_{I_2}(f_{\ell}; {\cal S}^2))]\\
&\;\;= \frac{\ell^3}{4} 
\Big[  \int_{I_1} p_1(t_1)d t_1  \int_{I_2} p_1(t_2)d t_2 -\iint_{I_1 \times I_2} g_2(t_1,t_2) d t_1 d t_2  +16  \int_{I_1} g_3(t_1)d t_1  \int_{I_2} g_3(t_2)d t_2  \Big]+O(\ell^{5/2}),
\end{align*}
where
\begin{align} \label{3:24}
p_1(t)&=\frac{1}{(2 \pi)^{3/2}} \int_{\mathbb{R}^2} (x_1 t \sqrt{8}-x_1^2-x_2^2) \exp\left\{-\frac  3 2 t^2\right\} \exp\{-\frac 1 2 (x_1^2+x_2^2-\sqrt 8 t x_1)\} d x_1 d x_2, 
\end{align}
\begin{align} \label{3:31}
g_2(t_1,t_2)&= \frac{1}{2} \frac{1}{(2 \pi)^{3}}\iint_{\mathbb{R}^2 \times
\mathbb{R}^2} \left( z_1 \sqrt 8 t_1  -z_1^2-z_2^2\right) \exp\left\{ - \frac{3 }{2} t_1^2 \right\}
\exp\left\{-\frac 1 2 (z_1^2+z_2^2- \sqrt 8 t_1 z _1)\right\} \\
&\;\; \times \left( w_1 \sqrt 8 t_2  -w_1^2-w_2^2\right) \exp\left\{ - \frac{3 }{2} t_2^2 \right\}
\exp\left\{-\frac 1 2 (w_1^2+w_2^2- \sqrt 8 t_2 w _1)\right\}  \nonumber \\
&\;\; \times \left[-6+(3 t_1-\sqrt{2} z_1)^2+(3 t_2 -\sqrt{2} w_1)^2\right] d z_1 d z_2 \nonumber 
d w_1 d w_2,
\end{align}
and
\begin{align} \label{3:311}
g_3(t)= \frac{1}{8} \frac{1}{ (2 \pi)^{3/2}}\int_{\mathbb{R}^2} \left(z_1 \sqrt 8 t -z_1^2-z_2^2\right)
\exp\left\{- \frac 3{2} t^2\right\} \exp\left\{-\frac 1 2 (
z_1^2+z_2^2- \sqrt 8 t z _1)\right\} \left[3-(3 t -\sqrt{2} z_1)^2\right] d z_1 d z_2.
\end{align}
\end{proposition}

\begin{proof}
In view of Proposition \ref{lemma2} and by isotropy we have to study the
asymptotic behaviour of 
\begin{align}  \label{2:15}
16 \pi^2 \int_{C/\ell}^{\pi/2} \iint_{I_1 \times I_2} {J}_{2,\ell} (\phi;
t_1,t_2) d t_1 d t_2 \sin \phi \; d \phi- \mathbb{E}[\chi (A_{I_1}(f_{\ell}; 
{\mathcal{S}}^2))] \mathbb{E}[\chi (A_{I_2}(f_{\ell}; {\mathcal{S}}%
^2))]+O(\ell^{2}).
\end{align}
Again we stress that $J_{2,\ell}$ in \eqref{J2} is analogous to $K_{2,\ell}$
in \eqref{K2} except for the fact that the absolute value of the Hessian
determinant has been dropped (by means of Morse theorem).

The proof of this proposition follows along the same lines as in the
argument given in \cite{CMW} Section 4.1.2 where we study the asymptotic
behaviour of 
\begin{align*}
16 \pi^2 \int_{C/\ell}^{\pi/2} \iint_{I \times I} {K}_{2,\ell} (\phi;
t_1,t_2) d t_1 d t_2 \sin \phi \; d \phi- \big( \mathbb{E}[{\mathcal{N}}%
^c_I(f_{\ell})] \big)^2
\end{align*}
to obtain the variance of the number of critical points. Therefore here we
just sketch the main steps and we refer to \cite{CMW} Section 4.1.2 for a
complete proof.

The asymptotic analysis is based on the properties of multivariate
conditional Gaussian variables, and on an asymptotic study of the tail decay
of Legendre polynomials and their derivatives that appear in the conditional
covariance matrix of the Gaussian vector. In fact, for $d(x,y)>C/\ell$, $C$
large enough, Kac-Rice formula holds exactly and we one can exploit the fact
that a Gaussian expectation is an analytic function with respect to the
parameters of the corresponding covariance matrix outside its singularities.
It is then possible to compute the Taylor expansion of these expected values
around the origin with respect to the vanishing entries 
\begin{equation*}
\mathbf{a}=\mathbf{a}_{\ell}(\phi)=(a_{1,\ell}(\phi),
a_{2,\ell}(\phi),a_{3,\ell}(\phi),a_{4,\ell}(\phi),a_{5,\ell}(\phi),a_{6,%
\ell}(\phi),a_{7,\ell}(\phi),a_{8,\ell}(\phi))
\end{equation*}
of the conditional covariance matrix $\Delta_{\ell}(\phi)=\Delta(\mathbf{a})$
(see \cite{CMW} Appendix B) of the centred Gaussian random vector 
\begin{equation*}
\frac{\sqrt 8}{ \lambda_{\ell}} (\nabla ^{2}f_{\ell }(x),\nabla ^{2}f_{\ell
}(y)\big|\nabla f_{\ell }(x)=\nabla f_{\ell }(y)=\mathbf{0}).
\end{equation*}
Three terms in the Taylor expansion (depending on the intervals $I_1$ and $%
I_2$) give an asymptotically significant contribution, whereas the rest is
negligible: 
\begin{align}  \label{dominantttttt}
&16 \pi^2 \int_{C/\ell}^{\pi/2} \iint_{I_1 \times I_2} {J}_{2,\ell} (\phi;
t_1,t_2) d t_1 d t_2 \sin \phi \; d \phi- \mathbb{E}[\chi (A_{I_1}(f_{\ell}; 
{\mathcal{S}}^2))] \mathbb{E}[\chi (A_{I_2}(f_{\ell}; {\mathcal{S}}^2))] 
\notag \\
&= \ell^3 \Big\{ \frac{1}{4} \int_{I_1} p_{1} (t_1 ) d t_1 \int_{I_2}
p_{1}(t_2 ) d t_2 -16 \iint_{I_1 \times I_2} \Big[\frac{\partial }{\partial
a_{3} }q (\mathbf{a}; t_{1},t_{2})\Big]_{\mathbf{a}=\mathbf{0}} d t_1 d t_2 
\notag \\
&\;\;+32 \iint_{I_1 \times I_2} \Big[\frac{\partial ^{2}}{\partial a_{7}^2}q
(\mathbf{a}; t_{1},t_{2})\Big]_{\mathbf{a}=\mathbf{0}} d t_1 d t_2 \Big\} %
+O(\ell^{5/2}),
\end{align}
here we set 
\begin{align*}
q (\mathbf{a};t_1,t_2)&=\frac{1}{(2\pi )^{3}}\iint_{\mathbb{R}^{2}\times 
\mathbb{R}^{2}} \left (z_{1} \sqrt 8 t_1 -z_{1}^{2}-z_{2}^{2}\right) \cdot
\left( w_{1} \sqrt 8 t_2 -w_{1}^{2}-w_{2}^{2}\right) \\
& \hspace{0.2cm} \times \hat{q} (\mathbf{a};
t_{1},t_{2};z_{1},z_{2},w_{1},w_{2})dz_{1}dz_{2}dw_{1}dw_{2},
\end{align*}
with 
\begin{equation*}
\hat{q} (\mathbf{a};t_{1},t_{2};z_{1},z_{2},w_{1},w_{2}) =\frac{1}{\sqrt{%
\det (\Delta (\mathbf{a}))}}\exp \left\{-\frac{1}{2}%
v_{t_{1},t_{2}}(z_{1},z_{2},w_{1},w_{2})\Delta (\mathbf{a}%
)^{-1}v_{t_{1},t_{2}}(z_{1},z_{2},w_{1},w_{2})^{t}\right\},
\end{equation*}
and $p_{1}$ defined in \eqref{3:24}. Note that the zeroth order term in the
Taylor expansion cancels out with 
\begin{equation*}
\mathbb{E}[\chi (A_{I_1}(f_{\ell}; \mathcal{S}^2))] \mathbb{E}[\chi
(A_{I_2}(f_{\ell}; \mathcal{S}^2))]
\end{equation*}
that is of order $O(\ell^4)$. The expressions for $g_2$ and $g_3$ in %
\eqref{3:31} and \eqref{3:311} follow from the evaluation of the partial
derivatives in formula \eqref{dominantttttt}; once more we refer to \cite%
{CMW} Section 4.1.2 for details.
\end{proof}

\noindent We can now prove Theorem \ref{cov}.

\begin{proof}[Proof of Theorem \protect\ref{cov}]
We first write: 
\begin{align*}
\text{Cov}[ \chi (A_{I_1}(f_{\ell }; {\mathcal{S}}^2) ) , \chi
(A_{I_2}(f_{\ell }; {\mathcal{S}}^2) ) ]=\mathbb{E}[\chi(A_{I_1}(f_{\ell }; {%
\mathcal{S}}^2) ) \chi(A_{I_2}(f_{\ell }; {\mathcal{S}}^2) ) ]- \mathbb{E}%
[\chi (A_{I_1}(f_{\ell }; {\mathcal{S}}^2) )] \mathbb{E}[\chi
(A_{I_2}(f_{\ell }; {\mathcal{S}}^2) )]
\end{align*}
where, in view of \eqref{morse}, 
\begin{align*}
\mathbb{E}[\chi(A_{I_1}(f_{\ell }; {\mathcal{S}}^2) ) \chi(A_{I_2}(f_{\ell
}; {\mathcal{S}}^2) ) ]=\mathbb{E}\big[ \sum_{j, k=0}^2 (-1)^{j+k} \mu_{j;1}
(f_{\ell }) \mu_{k;2} (f_{\ell })\big] =\sum_{j, k=0}^2 (-1)^{j+k} \mathbb{E}%
[ \mu_{j;1}(f_{\ell }) \mu_{k;2}(f_{\ell })] .
\end{align*}
Now by Proposition \ref{lemma3} the covariance is asymptotic to 
\begin{align*}
&\text{Cov}[\chi (A_{I_1}(f_{\ell }; {\mathcal{S}}^2) ) \chi
(A_{I_2}(f_{\ell }; {\mathcal{S}}^2) ) ] \\
&=\frac{\ell^3}{4} \Big[ \int_{I_1} p_1(t_1)d t_1 \int_{I_2} p_1(t_2)d t_2
-\iint_{I_1 \times I_2} g_2(t_1,t_2) d t_1 d t_2 +16 \int_{I_1} g_3(t_1)d
t_1 \int_{I_2} g_3(t_2)d t_2 \Big]+O(\ell^{5/2}).
\end{align*}
Now define 
\begin{align*}
p_2(t)&=\frac{1}{(2 \pi)^{3/2}} \int_{\mathbb{R}^2} (3t-\sqrt{2} x_1)^2 (x_1
t \sqrt{8}-x_1^2-x_2^2) \exp\left\{-\frac 3 2 t^2 \right\} \exp\left\{-\frac
1 2 (x_1^2+x_2^2-\sqrt 8 t x_1)\right\} d x_1 d x_2,
\end{align*}
it is easy to see that the functions $g_2$ and $g_3$ in \eqref{3:31} and %
\eqref{3:311} can be rewritten as 
\begin{align*}
g_2(t_1,t_2)= - 3 p_1(t_1) p_1(t_2)+ \frac 1 2 p_2(t_1) p_1(t_2)+ \frac 1 2
p_1(t_1) p_2(t_2), \hspace{0.5cm} g_3(t)= \frac 3 8 p_1(t)- \frac 1 8 p_2(t).
\end{align*}
Moreover $p_1$ and $p_2$ can be explicitly computed and we have: 
\begin{align*}
p_1(t)&= \frac{\sqrt 2}{\sqrt \pi} (t^2-1) e^{-\frac {t^2} 2 }, \hspace{1cm}
p_2(t) = \frac{\sqrt 2}{\sqrt \pi} (t^4+t^2-4) e^{-\frac {t^2} 2 }.
\end{align*}
It follows that we can rewrite the coefficient of the leading term in the
following form: 
\begin{align}  \label{10:26}
& \int_{I_1} p_1(t_1)d t_1 \int_{I_2} p_1(t_2)d t_2 -\iint_{I_1 \times I_2}
g_2(t_1,t_2) d t_1 d t_2 +16 \int_{I_1} g_3(t_1)d t_1 \int_{I_2} g_3(t_2)d
t_2  \notag \\
&= {\mathcal{I}}_{1,1} {\mathcal{I}}_{2,1} - \iint_{I_1 \times I_2} [- 3
p_1(t_1) p_1(t_2)+ \frac 1 2 p_2(t_1) p_1(t_2)+ \frac 1 2 p_1(t_1) p_2(t_2)]
d t_1 d t_2  \notag \\
&+16 \int_{I_1} [\frac 3 8 p_1(t_1)- \frac 1 8 p_2(t_1)]d t_1 \int_{I_2}
[\frac 3 8 p_1(t_2)- \frac 1 8 p_2(t_2)] d t_2  \notag \\
&= {\mathcal{I}}_{1,1} {\mathcal{I}}_{2,1} - [ -3 {\mathcal{I}}_{1,1} {%
\mathcal{I}}_{2,1}+ \frac 1 2 {\mathcal{I}}_{1,2} {\mathcal{I}}_{2,1}+\frac
1 2 {\mathcal{I}}_{1,1} {\mathcal{I}}_{2,2} ] + 16 [ \frac 3 8 {\mathcal{I}}%
_{1,1}- \frac 1 8 {\mathcal{I}}_{1,2}] [\frac 3 8 {\mathcal{I}}_{2,1} -
\frac 1 8 {\mathcal{I}}_{2,2} ]  \notag \\
&= \frac {25}{4} {\mathcal{I}}_{1,1} {\mathcal{I}}_{2,1}- \frac 5 4 {%
\mathcal{I}}_{1,2} {\mathcal{I}}_{2,1} -\frac {5}{4} {\mathcal{I}}_{1,1} {%
\mathcal{I}}_{2,2}+ \frac 1 4 {\mathcal{I}}_{1,2} {\mathcal{I}}_{2,2},
\end{align}
where $\mathcal{I}_{i,j}$, for $i,j=1,2$, are given by 
\begin{align*}
{\mathcal{I}}_{i,j}=\int_{I_i} p_j(t) d t, \hspace{1cm} p_1(t)= \frac{\sqrt 2%
}{\sqrt \pi} (t^2-1) e^{-\frac {t^2} 2 }, \hspace{1cm} p_2(t)= \frac{\sqrt 2%
}{\sqrt \pi} (t^4+t^2-4) e^{-\frac {t^2} 2 }.
\end{align*}
Formula \eqref{10:26} can be further simplified as follows: 
\begin{align*}
\frac{1}{2 \pi} \int_{I_1} (-t_1^4+4 t_1^2-1) e^{-\frac{t_1^2}{2}} d t_1
\int_{I_2} (-t_2^4+4 t_2^2-1) e^{-\frac{t_2^2}{2}} d t_2.
\end{align*}
\end{proof}

\begin{proof}[Proof of Corollary \protect\ref{var}]
The asymptotic expression for the variance in \eqref{11:14_2} follows
immediately by setting $I_1=I_2=I$, and we have: 
\begin{align*}
\text{Var}[\chi (A_{I}(f_{\ell }; {\mathcal{S}}^2) )]&=\frac{\ell^3}{4} %
\Big[ \frac{1}{\sqrt{2 \pi}} \int_{I} (-t^4+4 t^2 -1) e^{-\frac {t^2} 2 } d
t \Big]^2+O(\ell^{5/2}).
\end{align*}
That is the statement of Theorem \ref{var}.
\end{proof}

\begin{proof}[Proof of Corollary \protect\ref{covcor}]
\noindent In the particular case where $I_1=[u_1,\infty)$ and $%
I_2=[u_2,\infty)$, we have the following explicit form for the leading term
of the covariance: 
\begin{align*}
{\mathcal{I}}_1 {\mathcal{I}}_2= u_1 u_2 (u_1^2-1) (u_2^2-1) e^{- \frac{u_1^2%
}{2} } e^{- \frac{u_2^2}{2} }.
\end{align*}

\noindent Also, for $I_1=I_2=[u,\infty)$, our expression reduces to 
\begin{align*}
\text{Var}[\chi (A_{u}(f_{\ell }; {\mathcal{S}}^2) )] &= \frac{\ell^3}{8 \pi 
} (u-u^3)^2 e^{- u^2 } +O(\ell^{5/2}) =\frac{\ell^3}{8 \pi }
(H_3(u)+2H_1(u))^2 e^{- u^2 } +O(\ell^{5/2}),
\end{align*}
as claimed.
\end{proof}

\end{document}